\documentclass[10pt, a4paper]{amsart}

\usepackage[T1]{fontenc}
\usepackage{microtype}

\usepackage{amsmath}								
\usepackage{amssymb}
\usepackage{amsthm}
\usepackage{amscd}
\usepackage{amsfonts}
\usepackage{mathabx}
\usepackage{stmaryrd}
\usepackage{csquotes}

\usepackage{extarrows}

\usepackage[bookmarks=true, colorlinks=true, linkcolor=blue!50!black,
citecolor=orange!50!black, urlcolor=orange!50!black, pdfencoding=unicode]{hyperref}
\usepackage{color}

\usepackage{tikz}									
\usetikzlibrary{matrix}
\usetikzlibrary{patterns}
\usetikzlibrary{matrix}
\usetikzlibrary{positioning}
\usetikzlibrary{decorations.pathmorphing}
\usetikzlibrary{cd}

\usepackage{nicefrac}
\usepackage{fullpage}
\usepackage{enumerate}

\newtheorem*{theorem*}{Theorem}

\newtheorem{maintheorem}{Theorem}[section]
\newtheorem{maincorollary}[maintheorem]{Corollary}
									
\newtheorem{theorem}{Theorem}[section]
\newtheorem{lemma}[theorem]{Lemma}
\newtheorem{proposition}[theorem]{Proposition}
\newtheorem{corollary}[theorem]{Corollary} 

\theoremstyle{definition}
\newtheorem{definition}[theorem]{Definition}
\newtheorem*{definition*}{Definition}
\newtheorem{example}[theorem]{Example}

\newtheorem{remark}[theorem]{Remark}

\newtheoremstyle{myitemstyle}						
	{}			
	{}			
	{}			
	{}			
	{}			
	{.}			
	{ }			
	{}			
\theoremstyle{myitemstyle}
\newtheorem{myitemthm}{}

\newcommand{\R}{\mathbb{R}}

\newcommand{\Z}{\mathbb{Z}}

\newcommand{\Q}{\mathbb{Q}}

\newcommand{\PP}{\mathbb{P}}

\newcommand{\B}{\mathbb{B}}

\newcommand{\KK}{\mathbb{K}}

\newcommand{\calB}{\mathcal{B}}

\newcommand{\calI}{\mathcal{I}}

\newcommand{\calR}{\mathcal{R}}

\newcommand{\sfA}{\mathsf{A}}
\newcommand{\sfB}{\mathsf{B}}
\newcommand{\sfC}{\mathsf{C}}
\newcommand{\sfI}{\mathsf{I}}
\newcommand{\sfM}{\mathsf{M}}
\newcommand{\sfN}{\mathsf{N}}
\newcommand{\sfU}{\mathsf{U}}

\DeclareMathOperator{\id}{id}

\DeclareMathOperator{\rk}{rk}
\DeclareMathOperator{\nul}{nul}

\DeclareMathOperator{\BMat}{BMat}

\DeclareMathOperator{\supp}{supp}

\DeclareMathOperator{\Bond}{Bond}
\DeclareMathOperator{\vol}{vol}
\DeclareMathOperator{\Hess}{Hess}

\let\nul\relax
\DeclareMathOperator{\nul}{nul}

\DeclareMathAlphabet{\mymathbb}{U}{bbold}{m}{n}

\makeatletter
\newcommand{\superimpose}[2]{{\ooalign{$#1\@firstoftwo#2$\cr\hfil$#1\@secondoftwo#2$\hfil\cr}}}
\makeatother

\renewcommand{\hat}{\widehat}     
\newcommand{\bigmid}{\mathrel{\big|}}

\title{Logarithmic concavity of bimatroids} 

\date{}

\author{Felix R\"ohrle}
\address{Universit\"at T\"ubingen, Fachbereich Mathematik,  72076 T\"ubingen, Germany}
\email{\href{mailto:roehrle@math.uni-tuebingen.de}{roehrle@math.uni-tuebingen.de}}

\author{Martin Ulirsch}
\address{Universität Paderborn,
Institut für Mathematik,
33098 Paderborn, Germany}
\email{\href{mailto:ulirsch@math.uni-paderborn.de}{ulirsch@math.uni-paderborn.de}}

\begin{document}

\maketitle

\begin{abstract} 
A bimatroid is a matroid-like generalization of the collection of regular minors of a matrix. In this article, we use the theory of Lorentzian polynomials to study the logarithmic concavity of natural sequences associated to bimatroids. Bimatroids can be used to characterize morphisms of matroids and this observation (originally due to Kung) allows us to prove a weak version of logarithmic concavity of the number of bases of a morphism of matroids. This is weaker than the original result by Eur and Huh; it nevertheless  provides us with a new perspective on Mason's log-concavity conjecture for independent sets of matroids. We finally show that for realizable bimatroids, the regular minor polynomial is a volume polynomial. Applied to morphisms of matroids, this shows that the weak basis generating polynomial of a morphism is a volume polynomial; this confirms a conjecture of Eur--Huh for morphisms of nullity $\leq 1$ and gives an algebro-geometric explanation for Mason's log-concavity conjecture in the realizable case.  
\end{abstract}

\setcounter{tocdepth}{1}
\tableofcontents


\section*{Introduction}

Let $a_1,\ldots, a_s$ be a sequence of non-negative real numbers. Assume that there are no \emph{internal zeros}, i.e., if $a_i \neq 0$ and $a_j \neq 0$ for $i < j$ then $a_k \neq 0$ for all $i < k < j$. In this situation we say (in decreasing generality) that the sequence is 
\begin{itemize}
\item \emph{unimodal}, if there is an index $t$ such that $a_1\leq \cdots \leq a_t\geq \cdots \geq a_s$;
\item \emph{log-concave}, if we have
\begin{equation*}
    a_k^2\geq a_{k-1}\cdot a_{k+1}
\end{equation*}
for all $1<k<s$; and 
\item \emph{ultra log-concave}, if the sequence $\widetilde{a}_k:=\nicefrac{a_k}{{s\choose k}}$ is log-concave.
\end{itemize}

(Ultra-)log-concave and unimodal sequences are objects of central interest throughout many fields of mathematics, since showing these properties typically requires the combination of insights coming from \textit{a priori} quite different fields of mathematics. Recent examples are the proofs of log-concavity for the coefficients of the characteristic polynomial of a matroid in \cite{AdiprasitoHuhKatz} and for the number of independent sets of a matroid in \cite{Lenz, AdiprasitoHuhKatz, BrandenHuh, HuhSchroeterWang} (also see \cite{ChanPak} for further refinements), both of which require an intricate combination of methods from combinatorics, algebraic geometry and, in particular, from Hodge theory. We particularly highlight the recent introduction of the theory of Lorentzian polynomials in \cite{BrandenHuh} (also see \cite{AnariLiuGharanVinzantIII,AnariLiuGharanVinzantII, AnariLiuGharanVinzantI} for an equivalent approach using so-called \enquote{homogeneous completely log-concave polynomials}), which may be thought of as a generalization of both ultra log-concave sequences and the Hodge index theorem to the realm of homogeneous polynomials. 

In this article we expand on the developments in \cite{BrandenHuh} and apply these methods to \emph{bimatroids} in the sense of \cite{Kung_bimatroids} (or, equivalently, to \emph{linking systems} in the sense of \cite{Schrijver_linkingsystems}). This will give us a new perspective on the log-concavity results for morphisms of matroids proved in \cite{EurHuh}. In particular, we give a different perspective on previous proofs of Mason's log concavity conjecture for independent sets of a matroid and prove \cite[Conjecture 5.6]{EurHuh} on volume polynomials in the special case when the morphisms have nullity $\leq 1$. 

At this point, we would like to refer the reader to the excellent survey articles \cite{Katz_matroids, AdiprasitoHuhKatz_survey, Baker_Hodge,  Huh_ICM18, Huh_tropicalgeometryofmatroids, BrandenHuh_selfcontained, Huh_ICM22,   Eur_survey}, which have significantly shaped our understanding of the subject. After this preprint has appeared on the arXiv, in \cite{GRSU}, the second author and his collaborators have found generalizations of Theorem \ref{mainthm_bimatroidregularminors=ultralogconcave} and Corollary \ref{maincor_weakMason}, but not of Theorem \ref{mainthm_logconcavemorphism}, to the setting of valuated bimatroids respectively valuated matroids.

\subsection*{Acknowledgements}
The authors would like to thank Jeffrey Giansiracusa, Kevin K\"uhn, Arne Kuhrs, Felipe Rinc\'on,  Victoria Schleis, Benjamin Schr\"oter, Pedro Souza, Hendrik {S\"u\ss}, and Igor Pak
for helpful conversations and discussions en route to this article. Central ideas for this article were conceived during the 2023 Chow Lectures at the MPI in Leipzig, given by June Huh. We thank all speakers and organizers at this occasion. Finally, we thank the anonymous referee for their careful reading of the manuscript and their thoughtful comments.

\subsection*{Funding} This project has received funding from the Deutsche Forschungsgemeinschaft (DFG, German Research Foundation) TRR 195 \emph{Symbolic Tools in Mathematics and their Application}, from Deutsche Forschungsgemeinschaft (DFG, German Research Foundation), TRR 326 \emph{Geometry and Arithmetic of Uniformized Structures}, project number 444845124, from, and TRR 358 \emph{Integral Structures in Geometry and Representation Theory}, project number 491392403, as well as from the Deutsche Forschungsgemeinschaft (DFG, German Research Foundation) Sachbeihilfe \emph{From Riemann surfaces to tropical curves (and back again)}, project number 456557832, and the DFG Sachbeihilfe \emph{Rethinking tropical linear algebra: Buildings, bimatroids, and applications}, project number 539867663, within the  
SPP 2458 \emph{Combinatorial Synergies}.

\subsection*{Notation} For a set $S$ we write $2^S$ for its power set, i.e., the set of all its subsets, and we write ${S\choose k}$ and ${S\choose \leq k}$ for the set of all subsets of size $k$ and $\leq k$, respectively. Given two sets $S$ and $S'$ as well as $s\in S-S'$ and $s'\in S'-S$, we abbreviate $S-\{s\}\cup \{s'\}$ by $S_{s\leftrightarrow s'}$.


\section{Summary of the main results}

Let $E$ and $F$ be finite sets and let $A\in \KK^{E\times F}$ be a matrix over a field $\KK$ indexed by $E\times F$. It is well-known that the collection of linearly independent subsets of the set of column vectors of $A$ carries the structure of a (realizable) \emph{matroid} on the ground set $F$. 

\subsection{Regular minors} A finer invariant is the collection of \emph{regular minors} of $A$, i.e., the collection of $(I,J)\in{E\choose k}\times {F\choose k}$ for $k=0,\ldots, \min\big\{|E|,|F| \big\}$, such that $\det [A]_{I,J}\neq 0$. The regular minors of $A$ define a matroid-like structure that was introduced by Kung in \cite{Kung_bimatroids}, a so-called (realizable) \emph{bimatroid} (see Section \ref{section_bimatroidaxioms} below for a precise definition in the general, not necessarily realizable case). Around the same time, and apparently unbeknownst to each other, Schrijver introduced the mathematically equivalent notion of a \emph{linking system} in \cite{Schrijver_linkingsystems}. In this article we follow the terminology introduced in \cite{Kung_bimatroids} and give references to the corresponding results in \cite{Schrijver_linkingsystems}, whenever appropriate.

In the realizable case we consider the extended matrix $\widehat{A}=[I_E\vert A]\in \KK^{E\times (E\sqcup F)}$, where $I_E\in \KK^{E\times E}$ denotes the identity matrix of size $|E|$. The bimatroid of regular minors of the matrix $A$ admits a cryptomorphic description as the matroid of columns of $\widehat{A}$ (see Proposition \ref{prop_extendedmatroid} below for the general, not necessarily realizable case). 

An immediate application of Lorentzian polynomials to bimatroids implies our first result. 

\begin{maintheorem}\label{mainthm_bimatroidregularminors=ultralogconcave}
Let $E$ and $F$ be finite sets and $\sfA$ a bimatroid 
on the ground set $E\times F$. Set $m=\vert E\vert$ as well as $n=\vert F\vert$ and write $m\wedge n=\min\{m,n\}$. The number $R_k(\sfA)$ of regular $k \times k$ minors is an ultra log-concave sequence, i.e., we have 
\begin{equation*}
\frac{R_k(\sfA)^2}{{m\wedge n\choose k}^2}\geq \frac{R_{k+1}(\sfA)}{{m\wedge n\choose k+1}}\cdot \frac{R_{k-1}(\sfA)}{{m\wedge n\choose k-1}}
\end{equation*} 
for all $k\geq 1$. 
\end{maintheorem}

 We point out that we already have $R_k(\sfA)=0$, whenever $k$ is bigger than the rank of $\sfA$. In the realizable case, Theorem \ref{mainthm_bimatroidregularminors=ultralogconcave} tells us that the number $R_k(A)$ of regular $k\times k$ minors of a matrix $A\in \KK^{E\times F}$ is an ultra log-concave sequence.

\subsection{Regular rectangles} In Section \ref{section_regularrectangles} below we introduce a new cryptomorphic description of a bimatroid $\sfA$ in terms of \emph{(horizontal or vertical) regular rectangles}. In the realizable case, the regular rectangles of $A\in \KK^{E\times F}$ are precisely the (not necessarily square) submatrices $[A]_{S,T}$ for $S\subseteq E$ and $T\subseteq F$ that have maximal rank. We say that a regular rectangle is \emph{horizontal}, if $\vert S\vert \leq \vert T\vert$ and \emph{vertical}, if $\vert T\vert \leq \vert S\vert$. 

As an application of the strongest form of the Lorentzian property of the homogeneous independent set generating polynomials used in the proof of Mason's conjecture \cite{BrandenHuh}, we obtain the following result. 

\begin{maintheorem}\label{mainthm_rectanglesofmaximalrank=ultralogconcave}
Let $E$ and $F$ be finite sets and $\sfA$ a bimatroid on the ground set $E\times F$ and write $N=\vert E\vert +\vert F\vert$. Denote 
\begin{itemize}
\item by $RR^\updownarrow_k(\sfA)$ be the number of vertical regular rectangles in $\sfA$ with $k$ columns and 
\item by $RR^\leftrightarrow_k(\sfA)$  the number of horizontal regular rectangles in $\sfA$ with $k$ rows.
\end{itemize}
Then both sequences $RR^\updownarrow_k(\sfA)$ and $RR^\leftrightarrow_k(\sfA)$ are ultra log-concave. To be precise, we have
    \begin{equation*} \frac{RR_k^\updownarrow(\sfA)^2}{{N \choose k}^2} \geq \frac{RR_{k-1}^\updownarrow(\sfA)}{{N \choose k-1}} \cdot \frac{RR_{k+1}^\updownarrow(\sfA)}{{N \choose k+1}} \end{equation*}
as well as 
    \begin{equation*} \frac{RR_k^\leftrightarrow(\sfA)^2}{{N \choose k}^2} \geq \frac{RR_{k-1}^\leftrightarrow(\sfA)}{{N \choose k-1}} \frac{RR_{k+1}^\leftrightarrow(\sfA)}{{N \choose k+1}} \end{equation*}
    for all $k\geq 1$. 
\end{maintheorem}

\subsection{Logarithmic concavity of morphisms} Let $F$ and $F'$ be finite sets and $\sfM$ and $\sfM'$ matroids on $F$ and $F'$, respectively. Recall that a map $\phi\colon F\rightarrow F'$ is said to be a \emph{morphism of matroids} if the pullback $\phi^\ast\sfM'$ is a quotient of $\sfM$ (see Section \ref{section_morphisms} below for a reminder on this and related notions). The \emph{nullity} of $\phi$ is the difference $\nul(\phi)=\rk(\sfM)-\rk(\phi^\ast\sfM')$. 

In \cite{EurHuh} the authors define a subset $T\subseteq F$ to be a \emph{basis} of a morphism $\phi$, if $T$ is contained in a basis of $\sfM$ and $\phi(T)$ contains a basis of $\sfM'$. Denote by $\calB_k(\phi)$ the set of bases of $\phi$ and by $B_k(\phi)$ the number of bases of $\phi$ of a fixed cardinality $k$. 
Motivated by the characterization of morphisms of matroids via bimatroids developed in \cite{Kung_bimatroids} (see Section~\ref{section_basesofmorphisms}), we can prove the following result.

\begin{maintheorem}\label{mainthm_logconcavemorphism}
Let $\phi\colon \sfM\rightarrow\sfM'$ be a morphism of matroids. Then the sequence $B_k(\phi)$ is log-concave, i.e., we have
\begin{equation*}
B_k(\phi)^2 \geq B_{k+1}(\phi)\cdot B_{k-1}(\phi)
\end{equation*} 
for all $k\geq 1$. 
\end{maintheorem}

We point out that in \cite[Theorem 1.3]{EurHuh} the authors prove that the sequence $B_k(\phi)$ itself is ultra log-concave, so our result is a weaker version of theirs. 
It has been noted in \cite{EurHuh} that, when $\sfM'$ is the uniform matroid $\sfU_{0,1}$, the set of bases of $\phi$ is equal to the set of independent subsets of $\sfM$; hence the ultra log-concavity of $B_k(\phi)$ implies the strongest version of Mason's conjecture on the ultra log-concavity of the number of independent sets of a matroid $\sfM$. 
In our case, Theorem \ref{mainthm_logconcavemorphism} implies the following weaker version of Mason's conjecture.

\begin{maincorollary}\label{maincor_weakMason}
    Let $\sfM$ be a matroid and denote by $I_k(\sfM)$ the number of independent sets in $\sfM$ of cardinality $k$. Then the sequence $I_k(\sfM)$ is log-concave. 
\end{maincorollary}

For proofs of the ultra log-concavity of the sequence, i.e., the strongest version of Mason's conjecture, we of course refer the reader once again to \cite{AnariLiuGharanVinzantIII, BrandenHuh, HuhSchroeterWang}. Further refinements of both Mason's conjecture and \cite[Theorem 1.3]{EurHuh} can be found \cite[Theorems 1.6 and 1.16]{ChanPak} and the latter also contains criteria for when the inequalities are equalities. 

\subsection{Volume polynomials and the Eur--Huh conjecture} Let $X$ be a normal irreducible projective variety of dimension $d$ over an algebraically closed field $\mathbb{K}$ and consider a collection $H$ of nef $\Q$-divisors $H_1,\ldots, H_n$ on $X$. Then, by \cite[Theorem 4.6]{BrandenHuh} the \emph{volume polynomial}
\begin{equation*}
\vol_H(w):=(w_1H_1+\cdots+w_nH_n)^d=\sum_{\vert \alpha\vert =d}\frac{d!}{\alpha!} (H_1^{\alpha_1}\cdots H_n^{\alpha_n}) \cdot w^\alpha
\end{equation*}
is Lorentzian. The converse does not hold; i.e., not every Lorentzian polynomial is a volume polynomial. 
It is therefore a natural question to ask whether a given Lorentzian polynomial is a volume polynomial. We refer the reader to \cite{EurLarson} for a more elaborate theory of volume polynomials associated to discrete polymatroids. 

Given a bimatroid $\sfA\in\BMat^{E\times F}$ of rank $r$, we may consider its \emph{homogeneous regular minor polynomial}
\begin{equation*}
p_{\calR(\sfA)}(w)=\sum_{(I,J)\in\calR(\sfA)} \prod_{e\in I^c}w_e\cdot\prod_{f\in J}w_f \ ,
\end{equation*}
which is Lorentzian, since it agrees with the basis generating polynomial of the associated extended matroid $\widehat{\sfA}$. Using the construction of matroid Schubert varieties, originally due to Ardila and Boocher \cite{ArdilaBoocher}, we find the following result.

\begin{maintheorem}\label{mainthm_regularminorpolynomial=volumepolynomial}
Let $\sfA$ be a bimatroid which is realizable over an algebraically closed field $\mathbb{K}$. Then the homogeneous regular minor polynomial is a volume polynomial.
\end{maintheorem}

Deciding whether a given Lorentzian polynomial is a volume polynomial tends to be a difficult question, which requires significant geometric insight. 
Since our proof of Theorem \ref{mainthm_logconcavemorphism} via bimatroids is more direct and does not involve a limit argument (unlike the one in \cite{EurHuh}), Theorem \ref{mainthm_regularminorpolynomial=volumepolynomial} implies the following result in the case that both $\sfM$ and $\sfM'$ as well as the morphism $\phi$ are realizable.

\begin{maintheorem}\label{mainthm_weakvolumepolynomial}
    Let $\sfM$ and $\sfM'$ be matroids on finite sets $F$ and $F'$ and suppose that $\phi\colon F\rightarrow F'$ defines a morphism such that $\sfM$, $\sfM'$, and $\phi$ are realizable over an algebraically closed field $\mathbb{K}$. Denote by $r$ the rank of $\sfM$ and by $\nul(\phi)$ the \emph{nullity} of $\phi$. Given an integer $\alpha\geq \nul(\phi)$ the \emph{$\alpha$-weak homogeneous basis generating polynomial}
    \begin{equation*}
        p^\alpha_{\calB(\phi)}(w):=\sum_{k\geq 0} \, \sum_{T\in \calB_k(\phi)}{\alpha\choose r-k}w_0^{r-k}\prod_{f\in T} w_f
    \end{equation*}
    is a volume polynomial.
\end{maintheorem}

We point out two interpretations of Theorem~\ref{mainthm_weakvolumepolynomial}. First, if we choose $\alpha=r$, then Theorem \ref{mainthm_weakvolumepolynomial} (together with \cite[Theorem 4.6]{BrandenHuh}) provides us with an algebro-geometric explanation for the validity of Theorem \ref{mainthm_logconcavemorphism} in the realizable case and, in particular, for Mason's conjecture on the log-concavity of independent sets, as stated in Corollary \ref{maincor_weakMason}.

Second, one may view Theorem \ref{mainthm_weakvolumepolynomial} as a weak version of \cite[Conjecture 5.6]{EurHuh}, which states that, in the realizable case, the \emph{homogeneous basis generating polynomial} 
    \begin{equation*}
        p_{\calB(\phi)}(w):=\sum_{k\geq 0} \, \sum_{T\in \calB_k(\phi)}w_0^{\vert F\vert-k}\prod_{f\in T} w_f
    \end{equation*}
of $\phi$ is a volume polynomial. 

Suppose that the nullity of $\phi$ is $\leq 1$, which means that $\calB_k(\phi)=\emptyset$ unless $k=\rk\sfM$ or $k=\rk \sfM-1$. In this case we may choose $\alpha=\nul(\phi)$ and we have $p_{\calB(\phi)}(w)=w_0^{\vert F\vert-r}\cdot p^{\nul(\phi)}_{\calB(\phi)}(w)$. This product is again a volume polynomial by Proposition \ref{prop_invariancevolumepolynomials} (i) and (ii) below. Theorem \ref{mainthm_weakvolumepolynomial} therefore confirms the Eur--Huh conjecture for morphisms of nullity $\leq 1$.


\section{Bimatroids -- the basic story}
In this section we recall the basic theory of bimatroids, as originally introduced in \cite{Kung_bimatroids} (also see \cite{Schrijver_linkingsystems} for an alternative approach under the name \emph{linking systems} and \cite{Murota_valuatedbimatroids} as well as \cite{GRSU} for a valuated version of this story). This part is mostly expository and is meant to introduce the notation used in the remainder of the article. In Section \ref{section_regularrectangles}, we introduce a new cryptomorphic characterization of bimatroids in terms of \emph{regular rectangles}.

\subsection{Regular minors and extended matroid} \label{section_bimatroidaxioms} Bimatroids are an abstraction of the combinatorial structure of the set of regular minors of a given matrix, just like matroids are an abstraction of the combinatorial structure of linear independence of the columns of a matrix.

\begin{definition}\label{def_regularminors}
    Let $E$ and $F$ be finite sets and write ${E\choose \ast}\times {F\choose \ast}$ for the union $\bigcup_{d\geq 0} \Big({E\choose d}\times {F\choose d}\Big)$. A \emph{bimatroid} $\sfA$ on the ground set $E\times F$ is given by a subset $\calR(\sfA)$ of ${E\choose \ast}\times {F\choose \ast}$, called the set of \emph{regular minors} of $\sfA$, which fulfils the following axioms:
    \begin{enumerate}
        \item For $d=0$, the pair $(\emptyset, \emptyset)$ is a regular minor. 
        \item Let $(I, J)$ and $(I', J')$ be regular minors of $\sfA$.
        \begin{enumerate}
            \item For every $i'\in I'-I$ at least one of the following holds:
            \begin{itemize}
                \item there is $i\in I-I'$ such that $(I_{i\leftrightarrow i'},J)$ is a regular minor or
                \item there is $j'\in J'-J$ such that $\big(I\cup\{i'\}, J\cup\{j'\}\big)$ is a regular minor. 
            \end{itemize}
            \item For every $j\in J-J'$ at least one of the following holds:
            \begin{itemize}
                \item there is $j'\in J'-J $ such that $(I, J_{j\leftrightarrow j'})$ is a regular minor, or
                \item there is $i\in I-I'$ such that $\big(I-\{i\}, J-\{j\}\big)$ is a regular minor. 
            \end{itemize}
        \end{enumerate}
    \end{enumerate}
\end{definition}

We may write $\calR(\sfA)=\bigcup_{d\geq 0}\calR_d(\sfA)$, where $\calR_d(\sfA)$ denotes the set of regular minors of size $d\times d$ for $0\leq d\leq \min\big\{\vert E\vert,\vert F\vert\big\}$. The set of bimatroids on the ground set $E\times F$ is denoted by $\BMat^{E\times F}$. 

There is a natural bijection between ${E\choose \ast}\times {F\choose \ast}$ and ${E\sqcup F\choose \vert E\vert}$, given by the association 
\begin{equation}\label{eq_fundamentalbijection}
(I,J)\longmapsto I^c\sqcup J \ . 
\end{equation}
Via this bijection we may identify the set $\calR(\sfA)$ of \emph{regular minors} of $\sfA$ on $E\times F$ with the set of bases of a matroid $\widehat{\sfA}$ on the ground set $E \sqcup F$. The matroid $\widehat{\sfA}$ is called the \emph{extended matroid} associated to $\sfA$. 

\begin{proposition}\label{prop_extendedmatroid}
Let $E$ and $F$ be finite sets. A subset of ${E\choose \ast}\times {F\choose \ast}$ is the set of regular minors of a bimatroid $\sfA$ if and only if under bijection \eqref{eq_fundamentalbijection} it corresponds to the set of bases of a matroid $\widehat{\sfA}$ on $E\sqcup F$ such that $E$ is a basis. 
\end{proposition}

In order to prove Proposition \ref{prop_extendedmatroid} it is enough to observe that the basis exchange axioms for $\widehat{\sfA}$ are precisely equivalent to the axioms defining a bimatroid.

\begin{example}[Realizable bimatroids]\label{example_realizablebimatroids}
Let $E$ and $F$ be finite sets, let $\mathbb{K}$ be a field and $A\in \mathbb{K}^{E\times F}$ a matrix. Then the set $\calR(A)$ of those square submatrices of $A$, whose determinant does not vanish, form the regular minors of a bimatroid on the ground set $E\times F$. 

One way to see this is to consider the extended matrix
\begin{equation*}
\widehat{A}=\big[I_E\vert A\big] \in \mathbb{K}^{E\times (E\sqcup F)}
\end{equation*}
and to note that the set of bases of the vector space $\mathbb{K}^E$ among the column vectors form the bases of a matroid on $E\sqcup F$. Given an element in ${E\sqcup F\choose \vert E\vert}$, written as $I^c\sqcup J$ for $(I,J)\in{E\choose \ast}\times {F\choose \ast}$, we find that
\begin{equation*} 
\det \big[\widehat{A}\big]_{E,I^c\sqcup J}=\det \big[A\big]_{I,J} \ .
\end{equation*}
So these bases are in natural one-to-one correspondence with the regular minors of $A$. 
\end{example}

\begin{example}[Relations] \label{example_relations}
Let $E$ and $F$ be finite sets and let $R$ be a \emph{relation} between $E$ and $F$, i.e., a subset of $E\times F$. Given $(I,J)\in{E\choose \ast}\times{F\choose \ast}$, a \emph{matching} between  $I$ and $J$ is the graph of a bijection between $I$ and $J$. The set of $(I,J)\in {E\choose \ast}\times{F\choose \ast}$, for which $R$ contains a matching between $I$ and $J$ defines a bimatroid $[R]\in\BMat^{E\times F}$. 

One way to show this is to note that a relation between $E$ and $F$ may be interpreted as a Boolean matrix $A[R]\in \B^{E\times F}$. Then the bimatroid $[R]$ may be described as the Stiefel matroid on $E\sqcup F$ associated to the extended Boolean matrix $\widehat{A}[R]=\big[I_E\vert A[R]\big]\in\B^{E\times (E\sqcup F)}$, where $I_E$ is the tropical identity matrix (see \cite[Section 3.1]{FinkRincon_Stiefel} for details). In fact, we may choose a sufficiently generic lift of the Boolean matrix $\widehat{A}[R]=\big[I_E\vert A[R]\big]$ (which is always possible over an infinite field) and apply Example \ref{example_realizablebimatroids}. 

Note that the graph of every map $\phi\colon F\rightarrow E$ is a relation between $E$ and $F$. So, in particular, we have an induced bimatroid $[\phi]\in\BMat^{E\times F}$.
\end{example}

\begin{example}[Bond bimatroids]\label{example_bondbimatroids}
Let $F$ be a finite set and $\sfM$ a matroid on $F$ of rank $r$. Choose a basis $B$ of $\sfM$. The \emph{bond bimatroid} $\Bond_B(\sfM)\in\BMat^{B\times F}$ of $\sfM$ with respect to the basis $B$ is defined as follows: A pair $(I,J)\in {B\choose\ast}\times{F\choose \ast}$ is a regular minor of $\Bond_B(\sfM)$ if and only if the (not necessarily disjoint) union $I^c\cup J$ is a basis of $\sfM$. 

A quick way to prove this is to observe that the regular minors of $\Bond_B(\sfM)$ naturally correspond to the set of bases of the matroid on the disjoint union $B\sqcup F$, where a subset of size $r$ is a basis if and only if its image in $B\cup F = F$ is a basis of $\sfM$. 
\end{example}

\subsection{Relative rank}
Let $E$ and $F$ be finite sets. A bimatroid $\sfA$ on $E\times F$ admits another cryptomorphic description in terms of the \emph{relative rank function} $r_{\sfA}\colon 2^E\times 2^F\rightarrow \Z_{\geq 0}$, which associates to $(S,T)\in 2^E\times 2^F$ the maximal number $d\geq 0$, for which there is a regular $d\times d$-minor $(I,J)\in\calR_d(\sfA)$ with $I\subseteq S$ and $J\subseteq T$. So we always have $0 \leq r(S, T) \leq \min\big\{|S|, |T|\big\}$ for all $(S,T)\in 2^E\times 2^F$. The number $r_{\sfA}(E,F)$ is called the \emph{rank} of the bimatroid $\sfA$. 

The relative rank provides us with a cryptomorphic description of bimatroids (see \cite[Section 5]{Kung_bimatroids} as well as \cite[(alternative) Definition 2.2]{Schrijver_linkingsystems}). 

\begin{proposition}\label{prop_relativerank}
Let $E$ and $F$ be two finite sets. A function $r\colon 2^E\times 2^F\rightarrow \Z_{\geq 0}$ is the relative rank function of a bimatroid $\sfA$ if and only if it fulfils the following properties:
\begin{enumerate}
\item For every $(S, T) \in 2^E \times 2^F$ we have $r(S, T) \leq \min\big\{ |S|, |T| \big\}$.
\item For every $(S,T)\in 2^E\times 2^F$ and $e\in E$ and $f\in F$, we have
\begin{equation*}
r(S,T)\leq r\big(S\cup\{e\},T\big)\leq r(S,T)+1 
\end{equation*}
as well as  
\begin{equation*}
 r(S,T)\leq r\big(S,T\cup\{f\}\big)\leq r(S,T)+1 \ .
\end{equation*}
\item For two pairs $(S,T), (S',T')\in 2^E\times 2^F$, we have
\begin{equation*}
r(S,T)+r(S',T')\geq r(S\cup S',T\cap T') + r(S\cap S', T\cup T') \ .
\end{equation*}
\end{enumerate}
\end{proposition}

\begin{proof}
Denote by $\widehat{r}$ the rank function of the extended matroid $\widehat{\sfA}$ on $E\sqcup F$. Then we have
\begin{equation*}
r(S,T)=\widehat{r}(S^c\sqcup T)-\vert S^c\vert 
\end{equation*}
for all $(S,T)\in  2^E\times 2^F$. The cryptomorphism here now follows from the cryptomorphic definition of matroids by their rank function. 
\end{proof}

\begin{remark}
    In \cite[Section 5]{Kung_bimatroids} Kung uses the normalization $r(\emptyset,\emptyset)=0$ instead of the axiom $r(S, T) \leq \min\big\{ |S|, |T| \big\}$. We believe that this is not correct, since otherwise the function $r(S,T)=\vert S\vert +\vert T\vert$ for $(S,T)\in 2^E\times 2^F$ would be the relative rank function of a bimatroid which does not match with the idea that the relative rank should be the size of the largest regular minor in $(S, T)$. To avoid this issue, we follow \cite[(alternative) Definition 2.2]{Schrijver_linkingsystems}.
\end{remark}

\begin{example}[Transpose of a bimatroid]
Let $E$ and $F$ be finite sets and $\sfA$ a bimatroid on $E\times F$. The \emph{transpose} $\sfA^T$ of $\sfA$ is a bimatroid on the ground set $F\times E$; a pair $(J,I)\in{F\choose\ast}\times{E\choose \ast}$ is a regular minor of $\sfA^T$ if and only if $(I, J)$ is a regular minor of $\sfA$. 
The relative rank function of $\sfA^T$ is given by 
\begin{equation*}
r_{\sfA^T}(T,S)=r_{\sfA}(S,T)
\end{equation*}
for $S\times T\subseteq E\times F$. One way to see that $\sfA^T$ is indeed a bimatroid is to observe that the associated extended matroid $\widehat{\sfA}^T$ is the dual matroid to $\widehat\sfA$. In terms of rank functions, this can be deduced from
\begin{equation*}\begin{split}
    \widehat{r}_{\sfA^T}(T^c\sqcup S)&= r_{\sfA^T}(T,S) + \vert T^c\vert\\
    &=r_{\sfA}(S,T) + \vert S^c\vert +\vert S\vert -\vert E\vert +\vert T^c\vert\\
    &=\widehat r_{\sfA}(S^c\sqcup T) + \big(\vert S\vert +\vert T^c\vert\big) - \vert E\vert  
\end{split}\end{equation*}
for every $(S,T)\in 2^E\times 2^F$ and the observation that the right hand side is the rank function of the matroid dual to $\widehat{\sfA}$.
\end{example}


\subsection{Regular rectangles}\label{section_regularrectangles}
Let $E$ and $F$ be finite sets and $\sfA$ a bimatroid on $E\times F$. Given two subsets $S\subseteq E$ and $T\subseteq F$, we say that the pair $(S,T)$ is a \emph{regular rectangular minor} (for short a \emph{regular rectangle}), if $r_\sfA(S,T)=\min\big\{\vert S\vert,\vert T\vert\big\}$ or, equivalently, if there are  $I\subseteq S$ and $J\subseteq T$ such that $\vert I\vert =\vert J\vert=\min\big\{\vert S\vert,\vert T\vert\big\}$ and $(I,J)\in\calR(\sfA)$. 
We call a pair $(S,T)\in 2^E\times 2^F$ a \emph{vertical regular rectangle}, if $r_\sfA(S,T)=\vert T\vert\leq \vert S\vert$, and a \emph{horizontal regular rectangle}, if $r_\sfA(S,T)=\vert S\vert\leq \vert T\vert$. 
We write $\calR\calR(\sfA)$ for the set of regular rectangles and denote the subsets of vertical and horizontal rectangles by $\calR\calR^\updownarrow(\sfA)$ and $\calR\calR^{\leftrightarrow}(\sfA)$, respectively.

Regular rectangles can be used to give another cryptomorphic description of a bimatroid $\sfA$. 

\begin{proposition}\label{prop_regularrectangle}
Let $E$ and $F$ be finite sets and $\sfA$ a bimatroid on the ground set $E\times F$. The association $(S,T)\mapsto S^c\sqcup T$ defined in \eqref{eq_fundamentalbijection} induces a bijection between $\calR\calR^{\updownarrow}(\sfA)$ and the independent sets of $\widehat{A}$.
\end{proposition}

Proposition \ref{prop_regularrectangle} immediately implies the following cryptomorphic characterization of bimatroids.

\begin{corollary}
Let $E$ and $F$ be finite sets. A subset $\calR\calR^\updownarrow$ of $2^E\times 2^F$ is the set of vertical regular rectangles of a bimatroid $\sfA$ on $E\times F$ if and only if the following axioms hold:
\begin{enumerate}
\item $(\emptyset, \emptyset) \in \calR\calR^\updownarrow$
\item For all $(S,T)\in \calR\calR^\updownarrow$ we have $\vert T\vert \leq \vert S\vert$. 
\item Given $(S,T)\in \calR\calR^\updownarrow$ as well as $S\subseteq S'\subseteq E$ and $T'\subseteq T$, the pair $(S',T')$ is also in $\calR\calR^\updownarrow$.
\item Suppose we are given $(S,T), (S',T')\in \calR\calR^\updownarrow$ with $\vert T\vert -\vert S\vert>\vert T'\vert -\vert S'\vert$ there is
\begin{itemize}
\item $s'\in S'-S$ such that $(S'-\{s'\})\times T'\in \calR\calR^\updownarrow$ or
\item $t\in T-T'$ such that $S'\times (T'\cup\{t\})\in \calR\calR^\updownarrow$.
\end{itemize}
\end{enumerate}
\end{corollary}

\begin{proof}
It is enough to note that all subsets of $E$ are always independent in $\widehat{\sfA}$. The axioms are mere translations of the independent set axioms of $\widehat{\sfA}$ using the bijection in Proposition \ref{prop_regularrectangle}. 
\end{proof}

In the proof of Proposition \ref{prop_regularrectangle} we will make use of the following bimatroidal generalization of the Laplace expansion for matrices. 

\begin{lemma}[Laplace expansion]\label{lemma_Laplaceexpansion} Let $E$ and $F$ be finite sets and $\sfA$ a bimatroid on $E\times F$. Consider a regular minor $(I,J)\in\calR(\sfA)$.
\begin{enumerate}
\item For every $j\in J$ there is $i\in I$ such that  $\big(I-\{i\}, J-\{j\}\big)$ is a regular minor of $\sfA$. 
\item For every $i\in I$ there is $j\in J$ such that  $\big(I-\{i\}, J-\{j\}\big)$ is a regular minor of $\sfA$. 
\end{enumerate}
\end{lemma}

\begin{proof}
Part (1) immediately follows from Defintion \ref{def_regularminors} Axiom (2) (b) applied with $(I',J')=(\emptyset,\emptyset)$. Part (2) follows with the same argument applied to $\sfA^T$.
\end{proof}

\begin{proof}[Proof of Proposition \ref{prop_regularrectangle}]
Note, that the association $(S,T)\mapsto S^c\sqcup T$ defines a natural bijection
\begin{equation*}
\big\{(S,T)\in 2^E\times 2^F\ \big\vert\ \vert S\vert\geq \vert T\vert\big\}\longrightarrow {E\sqcup F\choose \leq\vert E\vert} \ .
\end{equation*}
We now restrict this map to $\calR\calR^{\updownarrow}$. Observe that the image of the restricted map is precisely the set of independent subsets of $\widehat{\sfA}$. An independent set of $\widehat{\sfA}$ corresponds to a pair $(S,T)\in 2^E\times 2^F$ with $\vert S\vert \geq \vert T\vert $ such that there are $I\subseteq S$ and $T\subseteq J\subseteq F$ with $(I,J)\in\calR(\sfA)$. If $I=S$ or $J=T$, this is already a regular rectangle. If not, we may apply the Laplacian expansion from Lemma \ref{lemma_Laplaceexpansion} above (possibly several times) and find $I'\subseteq I$ such that $(I',T)\in\calR(\sfA)$.
\end{proof}

Working with $\sfA^T$ instead of $\sfA$ we get a similar axiomatic characterization of the horizontal regular rectangles of a bimatroid. We leave the details of this reformulation to the avid reader.



\section{Products of bimatroids and morphisms of matroids} 

In this section we recall the construction of  products of bimatroids from \cite{Kung_bimatroids, Schrijver_linkingsystems} as well as the basic terminology of morphisms of matroids (following \cite{EurHuh}). We refer the reader to \cite[Chapter 17]{Welsh_matroidtheory}, to  \cite{Kung_strongmaps}, and the recent categorical exploration in \cite{HeunenPatta} for more background on this notion. Section \ref{section_basesofmorphisms} contains the central construction for the proof of Theorem \ref{mainthm_logconcavemorphism} in the next section.

\subsection{Products of bimatroids}
The central new feature of bimatroids is that, just like for matrices, but unlike for matroids, one can form products. In order to motivate this construction we recall the generalized \emph{Cauchy--Binet formula}. Let $E$, $F$, and $G$ be finite sets and $\mathbb{K}$ be a field. Given two matrices $A\in \mathbb{K}^{E\times F}$ and $B\in \mathbb{K}^{F\times G}$ and a pair $(I,K)\in{E\choose \ast}\times{G\choose \ast}$, we have
\begin{equation*}
\det [A\cdot B]_{I,K}=\sum_{J\in{ F\choose \vert I\vert}} \pm \det [A]_{I,J}\cdot \det [B]_{J,K} \ .
\end{equation*}

The Cauchy--Binet formula tells us that a minor $[A\cdot B]_{I,K}$ can only be regular if there is $J\in{ F\choose \vert I\vert}$ such that both $[A]_{I,J}$ and $[B]_{J,K}$ are regular as well.

\begin{proposition}\label{prop_product=bimatroid}
Let $E$, $F$, and $G$ be finite sets and let $\sfA$ and $\sfB$ be bimatroids on the grounds sets $E\times F$ and $F\times G$, respectively. Then the set $\calR(\sfA\cdot\sfB)$ of those $(I,K)\in {E\choose \ast}\times{G\choose\ast}$, for which there is $J\in {F\choose \ast}$ such that both $(I, J)\in\calR(\sfA)$ and $(J, K)\in\calR(\sfB)$, defines the set of regular minors of a bimatroid $\sfA\cdot \sfB$ on the ground set $E\times G$. 
\end{proposition}

The bimatroid $\sfA\cdot\sfB$ is called the \emph{product} of $\sfA$ and $\sfB$. A proof of Proposition \ref{prop_product=bimatroid} via rank functions, which uses the matroid intersection theorem, can be found in \cite[Section 6]{Kung_bimatroids} and \cite[Theorem 3.5]{Schrijver_linkingsystems}. Alternatively, following Frenk's thesis \cite[Prop. 4.2.21]{Frenk_thesis}, we can also observe that the extended matroid $\widehat{\sfA\cdot \sfB}$ of $\sfA\cdot \sfB$ is given by
\begin{equation*}
\widehat{\sfA\cdot \sfB} = \Big(\widehat{\sfA}\oplus \mathsf{0}_G\vee \mathsf{0}_E\oplus \widehat{\sfB} \Big)\big/_F\ ,
\end{equation*}
where we write $\mathsf{0}_E$ and $\mathsf{0}_G$ for the trivial matroid on the ground sets $E$ and $G$, respectively, the symbols $\oplus$ and $\vee$ stand for the direct sum and the union of matroids, respectively (see \cite[Section 8.3]{Welsh_matroidtheory} for further details on the latter), and $\big/_F$ denotes the contraction of a matroid along $F$. 

\begin{example}
    Let $E$, $F$, and $G$ be finite sets and let $R$ be a relation between $E$ and $F$ and $S$ a relation between $F$ and $G$. Recall (e.g., from \cite[Definition 0.5]{HeunenVicary_textbook}) that the \emph{product} $R\cdot S$ is defined by 
    \begin{equation*}
    R\cdot S=\big\{(e,g)\in E\times G\ \big\vert\ \textrm{there is } f\in F \textrm{ such that } (e,f)\in R \textrm{ and } (f,g)\in S\big\} \ .
    \end{equation*}
    This product is compatible with the product of bimatroids, i.e., we have 
    \begin{equation*}
    [R\cdot S]=[R]\cdot [S] \ .
    \end{equation*}
    In particular, given two maps $\phi\colon F\rightarrow E$ and $\psi\colon G\rightarrow F$, we have 
    \begin{equation*}
    [\phi\circ \psi]=[\phi]\cdot [\psi] \ .
    \end{equation*}
\end{example}

\begin{example}
    Let $E$ and $F$ be finite sets and $\sfA$ a bimatroid on the ground set $E\times F$. Given subsets $E'\subseteq E$ and $F'\subseteq F$, we write $i\colon E'\hookrightarrow E$ and $j\colon F'\hookrightarrow F$ for the inclusion maps. Then the \emph{restriction} $\sfA\vert_{E'\times F'}$ of $\sfA$ to $E'\times F'$ is defined by
    \begin{equation*}
    \sfA\vert_{E'\times F'}=[i]^T\cdot \sfA\cdot [j]\in \BMat^{E'\times F'} \ .
    \end{equation*}
    According to this definition, an element $(I,J)\in{E'\choose \ast}\times{F'\choose \ast}$ is a regular minor of $\sfA\vert_{E'\times F'}$ if and only if it is a regular minor of $\sfA$. 
\end{example}

The avid reader may now immediately verify that bimatroids naturally form a category $\mathbf{BMat}$ whose objects are finite sets and in which morphisms $F\rightarrow E$ (for finite sets $E$ and $F$) are  bimatroids on the ground set $E\times F$ (with products as composition). 
Explicitly this means the following:
\begin{itemize}
\item The product is associative: For all bimatroids $\sfA\in\BMat^{E\times F}$,  $\sfB\in\BMat^{F\times G}$, and $\sfC\in\BMat^{G\times H}$ and finite sets $E$, $F$, $G$, and $H$ we have
    \begin{equation*}
        (\sfA\cdot\sfB)\cdot \sfC=\sfA\cdot (\sfB\cdot\sfC) \ .
    \end{equation*}
    \item The bimatroids $\sfI_E=[\id_E]$ (for a finite set $E$) serve as identity morphisms: For all bimatroids $\sfA\in\BMat^{E\times F}$ we have 
    \begin{equation*}
        \sfI_E\cdot\sfA=\sfA=\sfA\cdot I_F \ . 
    \end{equation*}
\end{itemize}

\begin{remark}
Taking transpose naturally endows the category $\mathbf{BMat}$ with a dagger structure (see \cite[Section 2.3]{HeunenVicary_textbook} for details on this notion). In our situation this means the following: Given finite sets $E$, $F$, and $G$, we have $\sfI_E^T=\sfI_E$ as well as $(\sfA^T)^T=\sfA$ and $(\sfA\cdot \sfB)^T=\sfB^T\cdot \sfA^T$ for all bimatroids $\sfA\in\BMat^{E\times F}$ and $\sfB\in\BMat^{F\times G}$.
\end{remark}


\subsection{Morphisms of matroids}\label{section_morphisms}
Let $F$ and $F'$ be finite sets. Recall from \cite{EurHuh} that, given two matroids $\sfM$ and $\sfM'$ on $F$ and $F'$, respectively, a map $\phi\colon F\rightarrow F'$ is called a \emph{morphism} from $\sfM$ to $\sfM'$ if one (and therefore all) of the following three equivalent conditions hold:
\begin{itemize}
\item For all $T_1\subseteq T_2\subseteq F$, we have
\begin{equation*}
r_{\sfM'}\big(\phi(T_2)\big) - r_{\sfM'}\big(\phi(T_1)\big)\leq r_\sfM(T_2) - r_\sfM(T_1) \ . 
\end{equation*}
\item If $T'\subseteq F'$ is a cocircuit of $\sfM'$, then $\phi^{-1}(T')$ is a union of cocircuits of $\sfM$. 
\item If $T'\subseteq F'$ is a flat of $\sfM'$, then $\phi^{-1}(T')$ is a flat of $\sfM$. 
\end{itemize}

\begin{example}[Realizable morphisms]
Let $\mathbb{K}$ be a field. Let $\sfM$ and $\sfM'$ be matroids on finite ground sets $F$ and $F'$, realized by vectors $(v_s)_{s\in F}$ and $(v'_{s'})_{s'\in F'}$ spanning vector spaces $V$ and $V'$. Suppose further that we have a map $\phi\colon F\rightarrow F'$ and a $\mathbb{K}$-linear map $\Phi\colon V\rightarrow V'$ such that $\Phi(v_s)=v'_{\phi(s)}$ for all $s\in F$. Then $\phi$ defines a morphism of matroids. Morphisms of this type are called \emph{realizable} over $\KK$. 
\end{example}

Given two matroids $\sfM$ and $\sfN$ on the same ground set $F$, we say that $\sfN$ is a \emph{quotient} of $\sfM$, if the identity map $\id_F$ is a morphism of matroids from $\sfM$ to $\sfN$. Let $\phi\colon F\rightarrow  F'$ be a map of finite sets. For a matroid $\sfM'$ on $F'$, the \emph{pullback matroid} $\phi^\ast\sfM'$ is a matroid on $F$, whose rank function $r\colon 2^F\rightarrow\Z_{\geq 0}$ is given by
\begin{equation*}
r_{\phi^\ast\sfM'}(T)=r_{\sfM'}\big(\phi(T)\big)
\end{equation*}
for $T\subseteq F$. Recall that by \cite[Lemma 2.4]{EurHuh} a map $\phi\colon F\rightarrow F'$ is a morphism between matroids $\sfM$ and $\sfM'$ on $F$ and $F'$, respectively if and only if $\phi^\ast\sfM'$ is a quotient of $\sfM$. We call the difference of ranks
\begin{equation*}
\nul(\phi):=\rk(\sfM)-\rk(\phi^\ast\sfM')
\end{equation*}
the \emph{nullity} of $\phi$.

Following \cite[Definition 1.1]{EurHuh}, we say that a subset $T\subseteq F$ is a \emph{basis} of the morphism $\phi$ if it is contained in a basis of $\sfM$ and $\phi(T)$ contains a basis of $\sfM'$. By \cite[Lemma 2.4]{EurHuh} the set $\calB(\phi)$ of bases of $\phi$ is either empty, when $\phi(E)$ does not span $\sfM'$, or, otherwise, equal to the set of bases of the quotient $\sfM\rightarrow \phi^\ast \sfM'$. There are two extreme cases: When $\phi$ is the identity morphism, then $\calB(\phi)$ is the set of bases of $\sfM$ and, when $\sfM'$ is the uniform matroid $\sfM'=\sfU_{0,1}$, then $\calB(\phi)$ is the set $\calI(\sfM)$ of independent sets of the matroid $\sfM$. The last observation explains why Theorem \ref{mainthm_logconcavemorphism} implies Corollary \ref{maincor_weakMason}.

\subsection{Classifying bases of morphisms}\label{section_basesofmorphisms}
In \cite[Theorem 4]{Kung_bimatroids}, Kung characterizes morphisms of matroids in terms of bimatroid multiplication. This motivates the following construction.

\begin{proposition} \label{prop_basesofmorphismsasbimatroids}
Consider two finite sets $F$ and $F'$, two matroids $\sfM$ on $F$ and $\sfM'$ on $F'$, as well as a map $\phi\colon F\rightarrow F'$ that is a morphism of matroids such that $\phi(F)$ spans $\sfM'$. Write $\sfN$ for the pullback $\phi^\ast\sfM'$, so that $\sfN$ is a matroid on $F$ and a quotient of $\sfM$. Let $Q$ be a finite set of size $r=\rk\sfM$ and set $\widetilde{F}=Q\sqcup F$. Then the collection of $\widetilde{B}\subseteq \widetilde{F}$ such that $\vert\widetilde{B}\vert=\rk\sfM$ and $T=\widetilde{B}\cap F$ is a basis of $\phi$ is the set of bases of a matroid $\widetilde{\sfM}_\phi$ on $\widetilde{F}$. 
\end{proposition}

\begin{proof}
By the Higgs factorization theorem (see e.g., \cite[Theorem 8.2.8]{White_encyclopedia} or \cite[Section 7.3]{Oxley_matroids}), there is a finite set $F_0$ containing $F$ as well as a matroid $\sfM_0$ on $F_0$ such that 
\begin{equation*}
\sfM_0\vert_{F}=\sfM \qquad \textrm{ and } \qquad \sfM_0/Q_0=\sfN
\end{equation*}
where we denote the complement $F_0-F$ by $Q_0$. Here, as detailed in \cite[Lemma 7.3.3]{Oxley_matroids}, we may always choose $\sfM_0$ (and $F_0$) in such a way that $Q_0$ is an independent subset of $\sfM_0$ and $\rk(\sfM_0)=\rk(\sfM)$. 

Note that in this case $\vert Q_0\vert=\nul(\phi)$ and a subset $T\subseteq F$ is a basis of $\phi$ if and only there is a subset $S_0\subseteq Q_0$ such that $S_0\sqcup T$ is a basis of $\sfM_0$ (we  automatically have $\vert S_0\vert=\rk\sfM-\vert T\vert$ in this situation). To prove the last assertion:
\begin{itemize}
\item Let $T\subseteq F$ be a basis of $\phi$. Then $T$ contains a basis $B$ of $\sfN=\sfM_0/Q_0$. By \cite[Cor. 3.1.8]{Oxley_matroids}, this means that $Q_0\sqcup B$ is a basis of $\sfM_0$. Since $T$ is contained in a basis of $\sfM$, it is independent in $\sfM$ and, thus, also in $\sfM_0$. We may now use the axioms of an independent set to successively add elements in $Q_0\sqcup B -T=Q_0$ to $T$ until we find a subset $S_0\subseteq Q_0$ such that $S_0\sqcup T$ is a basis of $\sfM_0$.

\item Let $T\subseteq F$ and suppose that there is $S_0\subseteq Q_0$ such that $S_0\sqcup T$ is a basis of $\sfM_0$. Then $T$ is independent in $\sfM_0$ and, thus, also in $\sfM$. Hence $T$ is contained in a basis of $\sfM$. Moreover, by \cite[Prop. 3.1.6]{Oxley_matroids}, we have 
\begin{equation*}
\rk_{\sfN}(T)=\rk_{\sfM_0/Q_0}(T)=\rk_{\sfM_0}(Q_0\sqcup T)-\rk_{\sfM_0}(Q_0) \ .
\end{equation*}
Since $Q_0$ is independent, we have $\rk_{\sfM_0}(Q_0)=\vert Q_0\vert=\nul(\phi)$, and, since $S_0\sqcup T\subseteq Q_0\sqcup T$ is a basis of $\sfM_0$, we have $\rk_{\sfM_0}(Q_0\sqcup T)=\rk(\sfM_0)=\rk(\sfM)$. This implies $\rk_{\sfN}(T)=\rk(\sfM)-\nul(\phi)=\rk(\sfN)$ and so $T$ must contain a basis of $\sfN$. 
\end{itemize}

Choose a basis $B_0$ of $\sfM_0$ such that $Q_0\subseteq B_0$ and consider the bond bimatroid $\Bond_{B_0}(\sfM_0)$. Denote by $R$ the relation on $B_0\times B_0$ that matches a pair $(I_0,J_0)\in {B_0\choose \ast}\times {B_0\choose \ast}$ if and only if $\vert I_0\cap Q_0\vert=\vert J_0\cap Q_0\vert$ and $I_0-Q_0=J_0-Q_0$. This can be realized by the relation $R=Q_0\times Q_0\cup \big\{(s,s) \mathrel{\big\vert} s\in B_0-Q_0 \big\}\subseteq B_0\times B_0$. Then the extended matroid $\widetilde{\sfM}_0$ of the bimatroid
\begin{equation*}
[R]\cdot \Bond_{B_0}(\sfM_0)\vert_{B_0\times (F-B_0)}
\end{equation*}
has the property that a subset $T\subseteq F$ is a basis of $\phi$ if and only if $S_0\sqcup T$ is a basis of $\widetilde{\sfM}_0$ for every subset $S_0\subseteq Q_0$ of size $r-\vert T\vert$.

We now use the basis exchange property of $\widetilde{\sfM}_0$ in order to show that the basis exchange property for $\widetilde\sfM_\phi$ holds. We need to verify that, given two bases $U=S\sqcup T$ and $U'=S'\sqcup T'$ of $\widetilde{\sfM}_\phi$ (with $S,S'\subseteq Q$ and $T,T'\subseteq F$) and an element $x\in U-U'$, there is $x'\in U'-U$ such that $U-\{x\}\cup\{x'\}$ is also a basis of $\widetilde\sfM_\phi$.

Consider first the case that $x\in S$. If $S'-S$ is not empty, we may choose $x'\in S'-S$ and $U-\{x\}\cup\{x'\}$ is also a basis of $\widetilde\sfM_\phi$. 

So now suppose that $S'\subseteq S$. Since $\vert U\vert =\vert U'\vert =\rk \sfM$ and $\vert T\vert,\vert T'\vert\geq \rk\sfN$, we may choose subsets $S_0\subseteq Q_0$ and $S_0'\subseteq S_0$ with $\vert S_0\vert=\vert S\vert$ and $\vert S_0'\vert=\vert S'\vert$ as well as a bijection $\gamma\colon S\xrightarrow{\sim} S_0$, such that $\gamma(S')=S_0'$. Set $x_0=\gamma(x)\in S_0-S_0'$ as well as $U_0=S_0\sqcup T$ and $U_0'=S_0'\sqcup T'$. Then both $U_0$ and $U_0'$ are bases of $\widetilde{\sfM}_0$, since $T$ and $T'$ are bases of $\phi$. Moreover, we have $x_0\in U_0-U_0'$. So, the basis exchange property of $\widetilde{\sfM}_0$ provides us with $x'\in U_0'-U_0$ such that $U_0-\{x\}\cup\{x'\}$ is a basis of $\widetilde{\sfM}_0$. Since $S_0'\subseteq S_0$, we automatically have $x'\in T'-T$. This means that $T\cup\{x'\}$ is a basis of $\phi$ and, thus, $U-\{x\}\cup\{x'\}=S-\{x\}\sqcup T\cup\{x'\}$ is a basis of $\widetilde\sfM_\phi$. 

Consider now the case that $x\in T$. When $\vert S'\vert\leq \vert S\vert$, we may use again that $\vert U\vert =\vert U'\vert =\rk \sfM$ and $\vert T\vert,\vert T'\vert\geq \rk\sfN$ in order to choose subsets $S_0\subseteq Q_0$ and $S_0'\subseteq S_0$ with $\vert S_0\vert=\vert S\vert$ and $\vert S_0'\vert=\vert S'\vert$. Since $T$ and $T'$ are bases of $\phi$, we know that $U_0=S_0\sqcup T$ and $U_0'=S_0'\sqcup T'$ are bases of $\widetilde{\sfM}_0$. By the basis exchange property of $\widetilde{\sfM}_0$, there is $x'\in U_0'-U_0$ such that $U_0-\{x\}\cup\{x'\}$ is a basis of $\widetilde{\sfM}_0$. Since $S_0'\subseteq S_0$, we automatically have $x'\in T'-T$. This means that $T-\{x\}\cup\{x'\}$ is a basis of $\phi$ and so $U-\{x\}\cup\{x'\}$ is a basis of $\widetilde{\sfM}_\phi$. 

Finally, we need to deal with the case $\vert S'\vert>\vert S\vert$. Here we use again that $\vert U\vert =\vert U'\vert =\rk \sfM$ and $\vert T\vert,\vert T'\vert\geq \rk\sfN$ in order to be able to choose $S_0'\subseteq Q_0$ with $\vert S_0'\vert=\vert S'\vert$ and $S_0\subseteq Q_0'$ with $\vert S_0\vert=\vert S\vert$. Since $T$ and $T'$ are bases of $\phi$, we automatically have that $U_0=S_0\sqcup T$ and $U_0'=S_0'\sqcup T'$ are bases of $\widetilde{\sfM}_0$. The basis exchange property of $\widetilde{\sfM}_0$ provides us with $x_0'\in U_0'-U_0$ such that $U_0-\{x\}\cup\{x_0'\}$ is a basis of $\widetilde{\sfM}_0$. If $x_0'\in T'$, this means that $T-\{x\}\cup\{x_0'\}$ is also a basis of $\phi$. In this case we can choose $x'=x_0'$ and deduce that $U-\{x\}\cup\{x'\}=S\sqcup T-\{x\}\cup\{x_0'\}$ is a basis of $\widetilde{\sfM}_\phi$. If $x_0'\in S_0'$, we find that $T-\{x\}$ also has to be a basis of $\phi$. We now use $\vert S'\vert >\vert S\vert$ to be able to choose an element $x'\in S'-S$. Since $T-\{x\}$ is a basis of $\phi$, we have that $U-\{x\}\cup\{x'\}=S\cup\{x'\}\sqcup T-\{x\}$ is also a basis of $\widetilde{\sfM}_\phi$.
\end{proof}








\section{Lorentzian polynomials and logarithmic concavity}

We begin by recalling the definition and basic properties of \emph{Lorentzian polynomials} from \cite{BrandenHuh}. Let $n$ and $d$ be non-negative integers and write $\Delta_n^d=\big\{\alpha\in\Z_{\geq 0}^n \bigmid \vert\alpha\vert:=\alpha_1+\cdots + \alpha_n=d\}$. Denote by $H_n^d\subseteq \R[w_1,\ldots, w_n]$ the set of homogeneous polynomials with real coefficients of degree $d$ in $n$ variables $w_1, \ldots, w_n$. Whenever convenient, we write a polynomial $p(w)\in H_n^d$ in the variables $w=(w_1,\ldots, w_n)$ as 
\begin{equation*}
    p(w)=\sum_{\alpha\in\Delta_n^d}a_\alpha w^\alpha
\end{equation*}
using multi-index notation $w^\alpha=w_1^{\alpha_1}\cdots w_n^{\alpha_n}$ for $\alpha=(\alpha_1,\ldots, \alpha_n)\in \Z_{\geq 0}^n$. Following this notational logic, we also write $\partial^\alpha p$ for 
\begin{equation*}
\partial^\alpha p=\Big(\frac{\partial}{\partial w_1}\Big)^{\alpha_1}\cdots \Big(\frac{\partial}{\partial w_n}\Big)^{\alpha_n}p 
\end{equation*}
as well as $\alpha!=\alpha_1!\cdots\alpha_n!$.

Denote by $P_n^d$ the open subset of polynomials in $H_{n}^d$, for which all coefficients $a_\alpha$ are positive. 

\begin{definition}[{\cite[Definition 2.1]{BrandenHuh}}] Set $\mathring{L}_n^0=P_n^0$, $\mathring{L}_n^1=P_n^1$, as well as
\begin{equation*}
    \mathring{L}_n^2=\big\{p\in P_n^2 \bigmid \Hess(p) \textrm{ has the Lorentzian signature } (+,-,\ldots,-) \big\}
\end{equation*}
and
\begin{equation*}
\mathring{L}_n^d=\big\{p\in P_n^d \bigmid \partial^\alpha p\in \mathring{L}_n^2 \textrm{ for all }\alpha\in\Delta_n^{d-2} \big\} \ .
\end{equation*}
Polynomials in $\mathring{L}_n^d$ are called \emph{strictly Lorentzian} and we define the space $L_n^d$ of \emph{Lorentzian polynomials} as the closure of $\mathring{L}_n^d$ in $H_n^d$. 
\end{definition}

In the following, we write $e_1,\ldots, e_n$ for the standard basis vectors of $\Z^n$. Recall now, e.g., from \cite{Murota_discreteconvexanalysis}, that a subset $S\subseteq \Z_{\geq 0}^n$ is said to be \emph{M-convex}, if the following \emph{exchange property} holds: For any $\alpha, \beta\in S$ and any index $i$ satisfying $\alpha_i>\beta_i$, there is an index $j$ such that $\alpha_j<\beta_j$ and $\alpha-e_i+e_j\in S$. Observe that every M-convex set $S$ is automatically contained in $\Delta_n^d$ for some $d$ and that, if $S$ is also contained in $\{0,1\}^n$, the elements in $S$ are the indicator vectors of bases of a matroid $\sfM$; here the exchange property of $S$ is the same as the symmetric basis exchange property of $\sfM$.

Lorentzian polynomials enjoy the following useful properties: 
\begin{itemize}
\item A homogeneous polynomial $p(x,y)=\sum_{k=0}^d a_kx^ky^{d-k}$ in two variables with non-negative coefficients is Lorentzian if and only if the sequence $a_k$ is ultra log-concave and has no internal zeros (see \cite[Example 2.26]{BrandenHuh}). 
\item For every $p=\sum_{\alpha\in\Delta_n^d}a_\alpha w^\alpha\in L_n^d$ the \emph{support} 
\begin{equation*}
    \supp(p)= \big\{\alpha\in\Delta_n^d\bigmid a_\alpha\neq 0 \big\}
\end{equation*}
is an M-convex set (see \cite[Theorem 2.25]{BrandenHuh}). Vice versa, given an M-convex set $S\subseteq\Delta_n^d$, the \emph{indicator polynomial}
\begin{equation*}
    p_S(w)=\sum_{\alpha\in S}\frac{w^\alpha}{\alpha!}
\end{equation*}
is Lorentzian (see \cite[Theorem 3.10]{BrandenHuh}). 

\item Linear coordinate changes: Given a Lorentzian polynomial $p(w)\in L_m^d$ and a matrix $A\in\R_{\geq 0}^{m\times n}$, the linear coordinate change $p(Az)$ is also Lorentzian (see \cite[Theorem 2.10]{BrandenHuh}).
\item Given two Lorentzian polynomials $p\in L_n^d$ and $q\in L_{n'}^{d'}$, their product $p\cdot q\in L_{n+n'}^{d+d'}$ is again a Lorentzian polynomial (see \cite[Corollary 2.32]{BrandenHuh}). 
\item Given a Lorentzian polynomial $p(w)=\sum_{\alpha\in\Delta_n^d}a_\alpha w^\alpha$ as well as $\kappa\in\Z_{\geq 0}^n$, the \emph{$\kappa$-truncations}
\begin{equation*}
    p_{\leq\kappa}(w)=\sum_{\alpha\leq \kappa}a_\alpha w^\alpha \qquad \textrm{ and }\qquad p_{\geq\kappa}(w)=\sum_{\alpha\geq \kappa}a_\alpha w^\alpha
\end{equation*}
are also Lorentzian (see \cite[Proposition 3.3]{RossSuessWannerer}). Here we write $\alpha\leq \beta$ for $\alpha,\beta\in \Z_{\geq 0}^n$ if and only if $\alpha_i\leq \beta_i$ for all $1\leq i\leq n$.
\end{itemize}




We are now ready to prove Theorems~\ref{mainthm_bimatroidregularminors=ultralogconcave}, \ref{mainthm_rectanglesofmaximalrank=ultralogconcave}, and~\ref{mainthm_logconcavemorphism}.

\begin{proof}[Proof of Theorem~\ref{mainthm_bimatroidregularminors=ultralogconcave}]
    Let $\sfA$ be a bimatroid  on the ground set $E\times F$ and set $m=\vert E\vert$ as well as $n=\vert F\vert$. 
    Recall that the basis generating polynomial $p_{\calB(\widehat \sfA)}$ of the associated extended matroid $\widehat \sfA$ is given by
    \[  p_{\calB(\widehat \sfA)}(\underline{x}, \underline{y}) = \sum_{B \in \mathcal{B}(\hat \sfA)} \prod_{e \in B \cap E} x_e \prod_{f \in B \cap F} y_f \ . \]
    By \cite[Theorem~3.10]{BrandenHuh} we have that $p_{\calB(\widehat \sfA)}$ is Lorentzian (it is the generating function of the M-convex set of bases of a matroid). Note that by Proposition \ref{prop_extendedmatroid} this polynomial is equal to the regular minor polynomial $p_{\calR(\sfA)}$ of $\sfA$. 
    
    Now apply the coordinate transformation $x_e = x$ for all $e \in E$ and $y_f = y$ for all $f \in F$ and use \cite[Theorem~2.10]{BrandenHuh} to conclude that 
    \[ q(x, y) =  \sum_{k = 0}^{m} R_k(\sfA) x^{m-k} y^k\]
     is Lorentzian as well. Lorentzian bivariate polynomials were characterized in \cite[Example~2.26]{BrandenHuh} and this implies that the sequence $R_k(\sfA)$ fulfils 
     \begin{equation*}
        \frac{R_k(\sfA)^2}{{m\choose k}^2}\geq \frac{R_{k+1}(\sfA)}{{m\choose k+1}}\cdot \frac{R_{k-1}(\sfA)}{{m\choose k-1}}
    \end{equation*} 
    for all $k\geq 1$. The same argument, applied to $\sfA^T$ shows 
    \begin{equation*}
        \frac{R_k(\sfA)^2}{{n\choose k}^2}\geq \frac{R_{k+1}(\sfA)}{{n\choose k+1}}\cdot \frac{R_{k-1}(\sfA)}{{n\choose k-1}} 
    \end{equation*} 
    for all $k\geq 1$.
    We now combine these two inequalities and find
    \begin{equation*}
        \frac{R_k(\sfA)^2}{{m\wedge n\choose k}^2}\geq \frac{R_{k+1}(\sfA)}{{m\wedge n\choose k+1}}\cdot \frac{R_{k-1}(\sfA)}{{m\wedge n \choose k-1}} 
    \end{equation*} 
    for all $k\geq 1$.    
\end{proof}

\begin{proof}[Proof of Theorem \ref{mainthm_rectanglesofmaximalrank=ultralogconcave}]
    Consider the homogenized independent set generating polynomial $p_{\calI(\widehat{\sfA})}$ of the extended matroid $\hat A$. In our situation it may be written as 
    \begin{equation*}
        p_{\calI(\widehat{A})}(\underline x, \underline y,z) = \sum_{S \in \calI(\hat A)} z^{N - |S|} \prod_{e \in  S\cap E} x_e \prod_{f \in S \cap F} y_f  \ .
    \end{equation*}
    Then $p_{\calI(\widehat{\sfA})}$ is Lorentzian by \cite[Theorem 4.10]{BrandenHuh} (see the explanation right after \cite[Theorem 4.14]{BrandenHuh}). Set $N=\vert E\vert +\vert F\vert$. 
    By Proposition \ref{prop_regularrectangle} above, the polynomial $p_{\calI(\widehat{\sfA})}$ can be rewritten as 
    \begin{equation*}
        p_{\calI(\widehat{\sfA})}(\underline{x},\underline{y},z)= \sum_{(I, J) \in \calR \calR^\updownarrow(\sfA)} z^{N - |I^c| - |J|} \prod_{e \in I^c} x_e \prod_{f \in J} y_f \ .
    \end{equation*}
    Now we apply the coordinate transformation $x_e = z$ for all $e \in E$ and $y_f = y$ for all $f \in F$ to obtain
    \begin{equation*}
        q(y,z) = \sum_{k = 0}^N RR_k^\updownarrow(\sfA) z^{N - k} y^k \ .
    \end{equation*}  
    The polynomial $q$ is still Lorentzian by \cite[Theorem 2.10]{BrandenHuh}. 
    Thus, by \cite[Example 2.26]{BrandenHuh}, the sequence of coefficients is ultra log-concave. 
    This proves the first part of the theorem. The second statement is shown the same way using $\sfA^T$ instead of $\sfA$. 
\end{proof}

\begin{proof}[Proof of Theorem \ref{mainthm_logconcavemorphism}]
    Let $\sfM$ and $\sfM'$ be matroids on ground sets $F$ and $F'$ and consider a map $\phi\colon F\rightarrow F'$ that defines a morphism from $\sfM$ to $\sfM'$. Write $r:=\rk(\sfM)$. When $\phi(F)$ does not span $\sfM'$, the set of bases of $\phi$ is empty and there is nothing to show. So, setting $\sfN=\phi^\ast \sfM'$, it is enough to consider a quotient $\sfN$ of $\sfM$ on $F$. 
    Let $\widetilde{\sfM}_\phi$ be the matroid on ground set $Q \sqcup F$ constructed in Proposition~\ref{prop_basesofmorphismsasbimatroids} and let 
    \[ p_{\calB(\widetilde{\sfM}_\phi)}(\underline w) = \sum_{\substack{S \subseteq Q, T \subseteq F \\ S \sqcup T \in \calB(\widetilde{\sfM}_\phi)}} \prod_{q \in S} w_q \cdot \prod_{f \in T} w_f \]
    its basis generating polynomial. This is Lorentzian by \cite[Theorem 3.10]{BrandenHuh}. Setting all variables $w_q$ with $q \in Q$ equal to $x$ we obtain the weak basis generating polynomial 
        \begin{equation*}
            p_{w\calB(\phi)}(x, \underline w)=\sum_{T\in\calB(\phi)}{r\choose |T|} x^{\vert E\vert-\vert T\vert}\prod_{f\in T}w_f
        \end{equation*} 
    of $\phi$ by Proposition \ref{prop_basesofmorphismsasbimatroids}. This is again Lorentzian by \cite[Theorem 2.10]{BrandenHuh}. We now set all other variables $w_f$ for $f\in F$ equal to $y$ and obtain the polynomial
    \begin{equation*}
        q(x,y)=\sum_{k\geq 0}{r\choose k} B_k(\phi) x^{\vert E\vert -k}y^k \ .
    \end{equation*}
    This polynomial is also Lorentzian by \cite[Theorem 2.10]{BrandenHuh} and so the ultra log-concavity of the sequence ${r \choose  k}\cdot B_k(\phi)$ follows by \cite[Example 2.26]{BrandenHuh}. But this means that $B_k(\phi)$ itself is log-concave.
\end{proof}


\section{Volume polynomials}

Let $X$ be a normal irreducible projective variety of dimension $d$ over an algebraically closed field $\mathbb{K}$ and consider a collection $H$ of nef $\Q$-divisors $H_1,\ldots, H_n$ on $X$. Then, by \cite[Theorem 4.6]{BrandenHuh} the \emph{volume polynomial}
\begin{equation*}
\vol_H(w):=(w_1H_1+\cdots+w_nH_n)^d=\sum_{\vert \alpha\vert =d}\frac{d!}{\alpha!} (H_1^{\alpha_1}\cdots H_n^{\alpha_n}) \cdot w^\alpha
\end{equation*}
is Lorentzian. In general, we say that a Lorentzian polynomial $p(w)\in L_n^d$ is a \emph{volume polynomial}, if there exists a normal irreducible projective variety of dimension $d$ over an algebraically closed field $\mathbb{K}$ as well as a collection $H$ of nef $\Q$-divisors $H_1,\ldots, H_n$ on $X$ such that $\vol_H(w)=p(w)$. 

The class of volume polynomials satisfies the following useful properties, which are probably well-known and have been communicated to us by H.~S\"uss.

\begin{proposition} \label{prop_invariancevolumepolynomials}
    \begin{enumerate}[(i)]
\item Given two volume polynomials $p\in L_n^d$ and $q\in L_{n'}^{d'}$, their product $p\cdot q\in L_{n+n'}^{d+d'}$ is again a volume polynomial. 
\item Given a volume polynomial $p(w)\in L_m^d$ and a matrix $A\in \Q_{\geq 0}^{m\times n}$ with non-negative rational entries, the linear transformation $p(Az)\in L_n^d$ is also a volume polynomial. 
\end{enumerate}
\end{proposition}

\begin{proof}
For Part (i) we are given two normal irreducible projective varieties $X$ and $X'$ as well as two collections $H=(H_1,\ldots, H_n)$ and $H'=(H'_1,\ldots, H'_{n'})$ of nef $\Q$-divisors on $X$ and $X'$, respectively. Then we may consider the collection $\widetilde{H}=(H_1\times X',\ldots, H_n\times X', X\times H_1', \ldots, X \times H_{n'}')$ of nef $\Q$-divisors on the product $X\times X'$ and find
\begin{equation*}
    \vol_{\widetilde{H}}(\widetilde{w})=\vol_H(w)\cdot \vol_{H'}(w') 
\end{equation*}
in the variables $\widetilde{w}=(w,w')$ with $w=(w_1,\ldots, w_n)$ and $w'=(w_1',\ldots, w'_{n'})$.

For Part (ii) we consider a normal irreducible projective variety $X$ as well as a collection $H=(H_1,\ldots, H_m)$ of nef $\Q$-divisors on $X$. For a matrix $A=[a_{ij}]_{\substack{1\leq i\leq m 
\\ 1\leq j\leq n}}\in \Q_{\geq 0}^{m\times n}$ with non-negative entries, the $\Q$-divisors $H_j^A=a_{1j}H_1+\cdots + a_{mj}H_m$ are nef for all $1\leq j\leq n$. We write $H^A$ for the collection of $\Q$-divisors $H^A_j$ and find
\begin{equation*}\begin{split}
    \vol_H(Aw)&=\bigg(\Big(\sum_{1\leq j\leq n}a_{1j}w_j\Big)H_1+\cdots+\Big(\sum_{1\leq j\leq n}a_{nj}w_j\Big)H_m\bigg)^d\\
    &= \bigg(w_1\Big(\sum_{1\leq i\leq m}a_{i1}H_i\Big)+\cdots +w_n\Big(\sum_{1\leq i\leq m}a_{in}H_i\Big)\bigg)^d\\
    &=\vol_{H^A}(w) \ .
\end{split}\end{equation*}
Therefore, $\vol_H(Aw)$ is again a volume polynomial. 
\end{proof}

A central source of volume polynomials are \emph{matroid Schubert varieties}, which have first been studied in \cite{ArdilaBoocher} for realizable matroids and whose geometry motivates the development of the singular Hodge theory of matroids in \cite{BHMPW}. We also refer the reader to  \cite{EurLarson} for a generalization of matroid Schubert varieties to the setting of (realizable) discrete polymatroids/M-convex sets. 

\begin{example}
Let $\mathbb{K}$ be an algebraically closed field. Let $\sfM$ be a realizable matroid of rank $r$ on a ground set $E$, realized by a set of vectors $\{v_e\}_{e \in E}$ spanning a finite-dimensional $\mathbb{K}$-vector space $V$. Without loss of generality, $V$ is spanned by the $v_e$. Consider the natural surjective linear map $\bigoplus \mathbb{K} v_e\rightarrow V$; its dual is an injective linear map $V^\ast \rightarrow \prod_{e\in E}\mathbb{K} v^\ast_e$. The closure of $V^\ast$ in the multiprojective space $(\PP^1)^E$ is called the \emph{matroid Schubert variety} of $\sfM$ and will be denoted by $X_\sfM$. The hyperplane divisors on every projective line define a collection of ample divisors $H=(H_e)_{e\in E}$ on $X_\sfM$. In \cite[Theorem 1.3 (c)]{ArdilaBoocher} the authors show that for every multiset $\{e_1,\ldots, e_r\}$ in $E$ with underlying set $S$ we have
\begin{equation*}
(H_{e_1}\cdots H_{e_r})=\begin{cases}
1 & \textrm{ if } S \textrm{ is a basis of $\sfM$}\\
0 & \textrm{ else. }
\end{cases}
\end{equation*}
This implies that for the \emph{basis generating polynomial} $p_\sfM(w)=\sum_{B\in\calB(\sfM)}\prod_{b\in B}w_b$ of $\sfM$ we have
\begin{equation*}
 p_\sfM(w) =\frac{1}{d!}\vol_H(w) \ .
\end{equation*}
So the basis generating polynomial of a matroid which is realizable over $\mathbb{K}$ is always a volume polynomial. 
\end{example}

We are now ready to prove Theorems \ref{mainthm_regularminorpolynomial=volumepolynomial} and \ref{mainthm_weakvolumepolynomial}.

\begin{proof}[Proof of Theorem \ref{mainthm_regularminorpolynomial=volumepolynomial}]
 Let $A\in \mathbb{K}^{E\times F}$ be a matrix. Apply \cite[Theorem 1.3 (c)]{ArdilaBoocher} to find that the basis generating polynomial $p_{\widehat{\sfA}}$ of the matroid $\widehat{\sfA}$ associated to the extended matrix $\widehat{A}=[I_E\vert A]$ is a volume polynomial. By Proposition \ref{prop_extendedmatroid} this is the same as the regular minor polynomial $p_{\sfA}$ of the bimatroid $\sfA$ associated to $A$. Thus Theorem \ref{mainthm_regularminorpolynomial=volumepolynomial} is proved.
\end{proof}

\begin{proof}[Proof of Theorem \ref{mainthm_weakvolumepolynomial}]
Let $\sfM$ and $\sfM'$ be matroids on ground sets $F$ and $F'$, realized by vectors $(v_f)_{f\in F}$ and $(v'_{f'})_{f'\in F'}$ spanning vector spaces $V$ and $V'$, respectively.  Suppose further that we have a map $\phi\colon F\rightarrow F'$ and a $\mathbb{K}$-linear map $\Phi\colon V\rightarrow V'$ such that $\Phi(v_f)=v'_{\phi(f)}$ for all $f\in F$. Then $\phi$ defines a morphism of matroids. 

When $\Phi$ does not surject onto $V'$, there are no bases of $\phi$ and there is nothing to show. Hence, we may assume that $\Phi$ is surjective. 

Denote by $U$ the kernel of $\Phi$ and let $\alpha\geq \nul(\phi)$ be an integer. Since $\mathbb{K}$ is infinite, we may generically choose a generating set $\{u_e\}_{e\in E}$ of $U$ with $\vert E\vert=\alpha$ such that, given a linearly independent collection of vectors $(v_f)_{f\in T}$ (for $T\subseteq F$), the images $\big(\Phi(v_f)\big)_{f\in T}$ span $V'$ if and only if for every subset $S\subseteq E$ with $\vert S\vert +\vert T\vert =\dim V$ the set $\{u_e, v_f\}_{e\in S, f\in T}$ is a basis of $V$. To see that such a choice is indeed possible, note that we only need to find the $(\alpha \cdot \dim V)$-many entries of the $u_e$ such that for all collections of vectors $(v_f)_{f\in T}$ with $\big(\Phi(v_f) \big)_{f \in T}$ spanning $V'$ and every $S \subseteq E$ of the correct size, $\det \big( (u_e | v_f)_{e \in S, f \in T} \big) \neq 0$. A vanishing determinant is a codimension 1 condition in the parameter space $V^\alpha$ and there are only finitely many of these conditions. With infinite $\mathbb{K}$ a choice of the $u_e$ avoiding all determinatal conditions simultaneously is always possible.

Denote by $\widetilde{\sfM}_\alpha$ the matroid associated to $\{u_e, v_f\}_{e\in E, f\in F}$. We  consider the basis generating polynomial 
\begin{equation*}
p_{\calB(\widetilde{\sfM}_\alpha)}(w)=\sum_{B\in\calB(\widetilde{\sfM}_\alpha )}\prod_{e\in B\cap E}w_e \cdot \prod_{f\in B\cap F} w_f
\end{equation*} 
of $\widetilde{\sfM}_\alpha$. This is a volume polynomial by the construction of matroid Schubert varieties for realizable matroids (i.e., by \cite[Theorem 1.3 (c)]{ArdilaBoocher}). Write $r$ for the rank of $\sfM$. Setting all variables $w_e$ for $e\in E$ equal to $w_0$ and noting that $\vert E\vert=\alpha$, we find the $\alpha$-weak basis generating polynomial 
\begin{equation*}
    p^\alpha_{\calB(\phi)}(w)=\sum_{k\geq 0}{\alpha\choose r-k}w_0^{r-k}\cdot\sum_{T\in\calB_k(\phi)}\prod_{f\in  T} w_f
\end{equation*} 
of the morphism $\phi$. This is a volume polynomial by Proposition \ref{prop_invariancevolumepolynomials} (ii) above. 
\end{proof}


\section{Additional data and conflict of interest statement}

There is no additional data to this article. There is no conflict of interest for any of the authors. 





\bibliographystyle{amsalpha}
\bibliography{biblio}{}


\appendix



\end{document}